\documentclass[reqno,12pt,letterpaper]{amsart}
\usepackage{amsmath,amssymb,amsthm,graphicx,mathrsfs,url}
\usepackage{subfigure}
\usepackage{accents} 
\usepackage[usenames,dvipsnames]{color}
\usepackage[colorlinks=true,linkcolor=Red,citecolor=Green]{hyperref}

\def\?[#1]{\textbf{[#1]}\marginpar{\Large{\textbf{??}}}}

\newtheorem{theo}{Theorem}
\newtheorem{prop}{Proposition}[section]

\newtheorem{lem}[prop]{Lemma}
\newtheorem{cor}[prop]{Corollary}
\theoremstyle{remark}
\newtheorem{rem}{Remark}

\numberwithin{equation}{section}

\DeclareMathOperator{\sgn}{sgn}

\DeclareMathOperator{\Sp}{Sp}

\renewcommand{\Re}{\operatorname{Re}}

\title[Effects of confinement for single-well potentials]
{Effects of confinement for single-well potentials}
\author{Oran Gannot}
\email{ogannot@math.berkeley.edu}
\address{Department of Mathematics, Evans Hall, University of California,
Berkeley, CA 94720, USA}

\begin{document}

\begin{abstract}
We study bound states generated by a unique potential minimum in the situation where the system is strongly confined to a bounded region containing the minimum (by imposing Dirichlet boundary conditions). In this case the eigenvalues of the confined system differ from those of the unconfined system by an exponentially small quantity in the semiclassical limit. An asymptotic expansion for this shift is established. The formulas are evaluated explicitly for the harmonic oscillator and an application to the Coulomb potential at a fixed angular momentum is given.
\end{abstract}

\maketitle

\section{Introduction}

We study semiclassical Schr\"odinger operators with potential $V$ on subsets of the line, where $V$ admits a unique global minimum. More precisely, $V$ is required to satisfy
\begin{enumerate} \itemsep2pt
\item $V \in C^\infty(\mathbb{R})$,
\item $V(0) = V'(0) = 0$ and $V(x) > 0$ for $x \neq 0$,
\item $\liminf_{|x| \rightarrow \infty} V(x) > 0$.
\end{enumerate}
Define the self-adjoint operator $P(h) = h^2 D_x^2 + V$ acting on $L^2(\mathbb{R})$. If $\Omega \subset \mathbb{R}$ is a bounded open interval, let $P_\Omega(h)$ denote the Dirichlet realization of $h^2 D_x^2 + V$ on $L^2(\Omega)$.

It is well known that $P(h)$ has $m_0$ eigenvalues in the interval $I(h) = [0,C_0 h]$, where $m_0$ is the largest integer such that $(2m_0 + 1)\sqrt{V''(0)/2} < C_0$ and $h$ is sufficiently small depending on $C_0$ \cite{Helffer:1983,Simon:1983} --- such eigenvalues are typically referred to as \emph{low lying}. In fact, there exists a bijection 
\begin{equation} \label{eq:harmonicapproximation}  \sigma: \Sp P(h) \cap I(h) \rightarrow \Sp \widetilde{P}(h) \cap I(h), \text{  satisfying } \sigma(\lambda) - \lambda = \mathcal{O}(h^2), 
\end{equation}
where $\widetilde{P}(h) = h^2 D_x^2 + \tfrac{V''(0)}{2}x^2$ is the harmonic oscillator with eigenvalues 
\[
\sqrt{\tfrac{V''(0)}{2}}(2m+1)h, \ m\in \{0,1,2,\ldots \}.
\]
This result, a version of which actually holds in any dimension, is originally due to Simon \cite{Simon:1983} and Helffer--Sj\"ostrand \cite{Helffer:1983}; see also \cite{Dimassi:1999,Helffer:1988} for textbook treatments.

Now assume that $0 \in \Omega$, so $\Omega$ contains the global minimum of $V$ strictly in its interior. Then \eqref{eq:harmonicapproximation} is also valid for $P_\Omega(h)$ replacing $P(h)$. Moreover, tunneling estimates imply that the low lying eigenvalues of $P_\Omega(h)$ differ from those of $P(h)$ by an exponentially small quantity \cite[Chap. 6]{Dimassi:1999}: there exists $\varepsilon > 0$ and a bijection 
\begin{equation} \label{eq:tunneling}  \tau: \Sp P(h) \cap I(h) \rightarrow \Sp P_\Omega(h) \cap I(h), \text{  satisfying } \tau(\lambda) - \lambda = \mathcal{O}(e^{-\varepsilon/h}).
\end{equation}
This is also originally due to Helffer--Sj\"ostrand \cite{Helffer:1983}, and is valid in any dimension. 

The main theorem of this paper provides an asymptotic expansion for $\tau(\lambda) - \lambda$. To formulate the first result, write $\lambda_m^0$ for the $m$'th eigenvalue of $P(h)$, and similarly let $\lambda_m^\Omega$ denote the $m$'th eigenvalue of $P_\Omega(h)$. Note that $\tau (\lambda_m^0) = \lambda_m^\Omega$.
\begin{theo} \label{theo:maintheo1}
Fix an integer $m\geq 0$ and $\Omega = (r_-,r_+)$ with $-\infty < r_- < 0 < r_+ < \infty$. Then there exists $h_0 = h_0(m)$ such that $h \in (0,h_0)$ implies
\begin{equation} \label{eq:shift}
\lambda^\Omega_m - \lambda^0_m = h^{\frac{1}{2} - m} \sum_\pm  e^{-2\phi(r_\pm)/h} s_\pm(h).
\end{equation}
Here
\[
\phi(x) = \sgn x \int_0^x \sqrt{V}(t) \,dt,
\]
and 
\[
s_\pm(h)\sim \sum_{j=0}^\infty s^\pm_j\, h^j ,\quad s^\pm_0 = \frac{2^{m+1}}{m! \, \pi^{\frac12}}\left(\sqrt{\frac{V''(0)}{2}}\right)^{m+\frac{1}{2}}  \sqrt{V(r_\pm)}\, a_0(r_\pm)^2,
\]
where
\[
a_0(x) = \lim_{\varepsilon \rightarrow 0}\, (\varepsilon \sgn x)^m \exp{\left(\int_{\varepsilon\sgn x}^{x} \frac{\sqrt{\tfrac{V''(0)}{2}}(2m + 1) - \phi''(t)}{2\phi'(t)} \,dt\right)}.
\]
\end{theo}

The method of proof also applies to certain operators arising from spherically symmetric potentials in higher dimensions. Consider the operator $\mathbf{Q}(h) =  -h^2 \Delta_{\mathbb{R}^3} + \mathbf{W}(\mathbf{x})$ on $L^2(\mathbb{R}^3)$, where $\mathbf{W}(\mathbf{x}) = W(|\mathbf{x}|)$ for some $W: \mathbb{R} \rightarrow \mathbb{R}$. At a fixed angular momentum $\ell$, the study of $\mathbf{Q}(h)$ is equivalent to that of the effective Hamiltonian 
\[
Q(\nu;h) = h^2 D_x^2 + h^2 (\nu^2 - 1/4) x^{-2} + W(x)
\]
on $L^2((0,\infty))$, where $\nu := \ell + 1/2$. In fact, the main result holds for any $\nu > 0$. The physical potential $W$ is assumed to satisfy properties analogous to $V$,
\begin{enumerate}
\setcounter{enumi}{5} \itemsep2pt
\item $W \in C^\infty([0,\infty))$ \label{itm:first},
\item $W(0) = W'(0) = 0$ and $W(x) > 0$ for $x > 0$,
\item $\liminf_{x \rightarrow \infty} W(x) > 0$,
\item $W^{(2k+1)}(0) = 0$ for $k \geq 0$. \label{itm:fifth}
\end{enumerate}
Note that the assumption \eqref{itm:fifth} (along with assumption \eqref{itm:first}) is equivalent to the smoothness of $\mathbf{W}$ defined by  $\mathbf{W}(\mathbf{x}) = W(|\mathbf{x}|)$. In any case, it is necessary for the main result. 

If $0 < \nu < 1$, then $h^2 D_x^2 + h^2 (\nu^2 - 1/4) x^{-2}$ is not essentially self-adjoint on $C_c^\infty((0,\infty))$; instead we consider the Friedrichs extension, which can be characterized as the unbounded operator associated to the quadratic form
\[
\mathcal{Q}(u) = \int_0^\infty |hD_xu + ih(1/2-\nu)x^{-1}u|^2 \, dx
\]
on $H^1_0((0,\infty))$. This is further equivalent to the boundary condition 
\begin{equation} \label{eq:boundarycondition}
\lim_{x\rightarrow 0} x^{\nu - 1/2}u(x) = 0,
\end{equation} 
see \cite{Everitt:2007}. Now if $\Lambda = (0,L)$ denotes a finite interval, define $Q_\Lambda(\nu;h)$ as the self-adjoint operator on $L^2(\Lambda)$ with Dirichlet boundary conditions at $x = L$ and the boundary condition \eqref{eq:boundarycondition} when $0 < \nu <1$.

Although perhaps lesser known, there are natural analogues $\sigma_\nu, \tau_\nu$ of $\sigma, \tau$ as in \eqref{eq:harmonicapproximation}, \eqref{eq:tunneling}: define the harmonic oscillator 
\[
\widetilde{Q}_\mu(\nu;h) = h^2 D_x^2 + h^2 (\nu^2 - 1/4) x^{-2} + \tfrac{W''(0)}{2} x^2,
\] 
with eigenvalues $2\sqrt{\tfrac{W''(0)}{2}}(2m + 1 + \nu)h, \ m\in \{0,1,2,\ldots \}$. Then substitute 
\[
P(h) \Longleftrightarrow Q(\nu;h); \quad P_\Omega(h) \Longleftrightarrow Q_\Lambda(\nu;h); \quad \widetilde{P}(h) \Longleftrightarrow \widetilde{Q}(\nu;h),
\]
in \eqref{eq:harmonicapproximation}, \eqref{eq:tunneling} to get the appropriate statements for $\sigma_\nu, \, \tau_\nu$. Writing $\lambda_m^0$ and $\lambda^\Lambda_m$ for the $m$'th eigenvalues of $Q(\nu;h)$ and $Q_\Lambda(\nu;h)$, the following analogue of Theorem \ref{theo:maintheo1} holds.

\begin{theo} \label{theo:maintheo2}
Fix an integer $m\geq 0$ and $\Lambda = (0,L)$ with $0 < L < \infty$. Then there exists $h_0 = h_0(m)$ such that $h \in (0,h_0)$ implies
\begin{equation} \label{eq:frobshift}
\lambda^\Lambda_m - \lambda^0_m = h^{-\nu-2m} e^{-2\phi(L)/h} s(\nu;h).
\end{equation}
Here,
\[
\phi(x) = \int_0^x \sqrt{W}(t) \,dt
\]
and 
\[
s(\nu;h) \sim \sum_{j=0}^\infty s_j(\nu)\, h^j , \quad s_0(\nu) =  \frac{4\sqrt{W(L)}}{\Gamma(1+m+\nu) m!}\left(\sqrt{\frac{W''(0)}{2}}\right)^{2m+1+\nu} L^{1+2\nu} a_0(L)^2,
\]
where
\[
a_0(x) = \lim_{\varepsilon \rightarrow 0} \varepsilon^{2m} \exp\left( \int_\varepsilon^x \frac{2\sqrt{\tfrac{W''(0)}{2}}(2m + 1 + \nu)- \phi''(t) - (2\nu+1)t^{-1}\phi'(t) }{2\phi'(t)} \,dt \right).
\]
\end{theo}

\subsection{Some applications}

The simplest application of Theorem \ref{theo:maintheo1} is to the linear harmonic oscillator confined to a symmetric interval. Evaluating \eqref{eq:shift} to first order, we obtain the following corollary.

\begin{cor} \label{cor:semiclassicalho}
Let $V(x) = x^2$, so that $\lambda^0_m = (2m+1)h$. Let $\Omega(R) = (-R, R)$. Then for $R^{-2} h$ sufficiently small depending on $m$,
\begin{equation} \label{eq:semiclassicalho}
\lambda^{\Omega(R)}_m  = (2m+1)h + \frac{h^{\frac{1}{2}-m} 2^{2+m}}{m! \pi^{\frac12}} R^{2m+1} e^{-R^2/h}(1+\mathcal O(R^{-2}h)).
\end{equation}
\end{cor}
\begin{proof}
Set $x = Ry$. If $u(x)$ is an eigenvector of $P_{\Omega(R)}(h)$ with eigenvalue $\lambda^{\Omega(R)}_m$ then $\tilde{u}(y) := u(x) = u(Ry)$ is an eigenvector of $P_{\Omega(1)}(R^{-2} h )$ with eigenvalue $R^{-2}\lambda^{\Omega(R)}_m$. It remains to apply Theorem \ref{theo:maintheo1} with the effective semiclassical parameter $R^{-2} h$.
\end{proof}

A rigorous study of the semiclassical harmonic oscillator on a finite interval $\Omega \ni 0$ was previously performed in Bolley--Helffer \cite[Appendix 3]{Bolley:1993} with Neumann boundary conditions on the boundary of $\Omega$.

Analogously, Theorem \ref{theo:maintheo2} may be applied to the isotropic harmonic oscillator at a fixed angular momentum. 
\begin{cor} \label{cor:frobsemiclassicalho}
Let $W(x) = x^2$, so that $\lambda^0_m = 2(2m+1+\nu)h$. Then for $L^{-2} h$ sufficiently small depending on $m$,
\begin{equation} \label{eq:frobsemiclassicalho}
\lambda^\Lambda_m  = 2(2m+1+\nu)h +  \frac{ 4 h^{-2m - \nu} L^{2(2m + 1 + \nu)} }{m! \, \Gamma(1+m+\nu)}   e^{-L^2/h} \left ( 1+ \mathcal O(L^{-2} h) \right ).
\end{equation} 
\end{cor}
\begin{proof}
The proof is identical to that of Corollary \ref{cor:semiclassicalho}.
\end{proof}

An interesting application of Corollary \ref{cor:frobsemiclassicalho} is to the Coulomb Hamiltonian at a fixed angular momentum $\ell$,
\begin{equation*} \label{eq:coulombhamiltonian}
h^2 D_y^2 + \frac{h^2 \ell(\ell+1)}{y^2} - \frac{Z}{y}.
\end{equation*}
With initial domain $C_c^\infty((0,\infty))$, this Hamiltonian is essentially self-adjoint for $\ell > 0$. When $\ell = 0$ the deficiency indices both equal one --- see \cite{Oliveira:2009} for an explicit description of all the self-adjoint extensions. In particular, imposing a Dirichlet boundary conditions at $x = 0$ gives a self-adjoint extension. With a Dirichlet boundary condition when $\ell = 0$, the corresponding operator is denoted by $H(\ell;h)$ for $\ell \geq 0$. It is well known that $H(\ell;h)$ is bounded from below, and has discrete spectrum in $(-\infty,0)$. The negative eigenvalues can be listed, 
\[
E_{\ell+1} < E_{\ell + 2} < \cdots < 0, \text{ where } E_n = -\frac{Z^2}{4 n^2h^2}, \quad n \geq \ell + 1.
\]
Now let $H_R(\ell;h)$ denote the self-adjoint operator with same action as $H(\ell;h)$ but with a Dirichlet boundary condition at $x = R$. Then $H_R(\ell;h)$ is also bounded below with discrete spectrum in $(-\infty,0)$ and the negative eigenvalues will be listed as
\[
E_{\ell + 1}(R) < E_{\ell + 2}(R) < \cdots < 0.
\]
The following result holds for the difference $E_{n}(R) - E_n$.

\begin{cor} \label{cor:hydrogenshift}
Fix $n \geq \ell +1$. For $R^{-1} h^2$ sufficiently small depending on $n$, 
\begin{equation} \label{eq:hydrogenshift}
E_n(R) = -\frac{Z^2}{4 n^2h^2} + \frac{ 2^{2n+1} h^{-4n - 2} R^{2n} }{n^{2n+3} (n-\ell - 1)! (n+ \ell)!} \left(\frac{2}{Z}\right)^{-2n-2}e^{-ZR/nh^2} \left( 1+ 
\mathcal O\left( h^2 R^{-1} \right) \right).
\end{equation}
\end{cor}

\begin{proof}By rescaling, it may be assumed that $Z=2$. For a negative number $E<0$, let $k = (-E)^{-1/2}$. The $k$-dependent mapping
\[
\left (y \mapsto f(y) \right) \mapsto \left(x \mapsto x^{-1/2} f\left( 2 k h^{-1} x^2 \right) \right)
\]
maps the kernel of $H_R(\ell;h) - E(R)$ onto the kernel of $Q_{\Lambda}(2\ell + 1;h) - 4 k(R)$ (taking into account boundary conditions near the origin), where $L^2 = 2 R k(R)^{-1}$. 

We would like to formally apply Corollary \ref{cor:frobsemiclassicalho} to the operator $Q_{\Lambda}(2\ell + 1;h)$ to find an expression for $4k(R)$ in terms of the eigenvalues of $Q(2\ell + 1;h)$. To do this, it must first be verified that $L^{-2} h \rightarrow 0$ as $R^{-1}h^2 \rightarrow 0$; this is not immediately obvious since $L$ depends implicitly on $k(R)$, which is what we are trying to calculate in the first place. However, one has the following a priori information:

Suppose that $k(R) > 0$ corresponds to the $m$'th negative eigenvalue of $H_R(\ell;h)$. The claim is that $h^{-1} k(R) = \mathcal{O}(1)$ as $h^{-2}R \rightarrow \infty$. To prove this, note that $k(R)$ is characterized by the fact that
\[ 
y = 2R h^{-1} k(R)^{-1} \text{ is the $(m+1)$'th positive zero of } M_{h^{-1}k(R),\ell + 1/2}(y),
\]
where $M_{\kappa,\mu}$ is the Whittaker $M$-function. If the claim did not hold, then $h^{-1}k(R) \rightarrow \infty$
along some sequence of $h^{-2} R$ tending to infinity. Now the $r$'th zero of $M_{\kappa,\mu}$ as $\kappa \rightarrow \infty$ is given by $\alpha_{\mu,r}\kappa^{-1} + \mathcal O(\kappa^{-3/2})$, where $\alpha_{\mu,r} > 0$ is fixed \cite{Gabutti:2001}. If $\alpha := \alpha_{\ell+1/2,m+1}$, then along this sequence
\[
\frac{2R}{ h k(R)} \sim \frac{\alpha h}{ k(R)},   
\]
which is a contradiction since it implies $R h^{-2} = \mathcal{O}(1)$.
 
This shows that a priori, $L^{-2}h = \mathcal{O}( h^2 R^{-1} )$, and hence $L^{-2}h \rightarrow 0$. Applying Corollary \ref{cor:frobsemiclassicalho}, we find that for $n \geq \ell + 1$,
\begin{equation} \label{eq:hydrogenimplicit}
k(R) = n h + \frac{  h^{n-2} L^{4n}}{(n-\ell-1)! \, (n + \ell)! }   e^{-L^2/h} \left ( 1+ \mathcal O\left(L^{-2} h \right) \right )  .
\end{equation}
Therefore as a first approximation
\[
k(R) = nh + \mathcal O\left(\exp\left(-R/Ch^2 \right)\right),
\] 
and hence 
\[
L^{-2}   = 2^{-1} k(R) R^{-1} = 2^{-1} nh R^{-1} \left ( 1+ \mathcal O\left (\exp(-R/C_1 h^2 \right) \right).
\]
Plugging this back into \eqref{eq:hydrogenimplicit},
\[
k(R) = nh + \frac{ 2^{2n} h^{-4n + 1} R^{2n} } {n^{2n}(n-\ell - 1)! (n+ \ell)! } e^{-2R/nh^2} \left( 1 + \mathcal O\left (h^2 R^{-1} \right) \right).
\]
Now solve for $E(R) = -k(R)^{-2}$ to get that
\[
E_n(R) = -(nh)^{-2} + \frac{ 2^{2n+1} h^{-4n - 2} R^{2n} }{n^{2n+3} (n-\ell - 1)! (n+ \ell)!} e^{-2R/nh^2} \left( 1+ \mathcal O\left (h^2 R^{-1} \right) \right).
\]
\end{proof}

\subsection{Historical remarks}
The study of confined quantum mechanical systems has a long tradition --- the reader is referred to the articles of Fr\"oman et al \cite{Froman:1987}, Aquino \cite{Aquino:2009} and references therein for a comprehensive overview and physical applications. However, it should be stressed that few rigorous results appear in these reviews. Historically, the case of a hydrogen atom confined in a spherical box was the first problem of this type to be considered. Some of the earliest works in this direction are due to Michels et al \cite{Michels:1937}, Sommerfeld and Welker \cite{Sommerfeld:1938}, de Groot and ten Seldam \cite{deGroot:1946} in the physics literature. 

The formula \eqref{eq:hydrogenshift} for hydrogen was previously derived in the works of Dingle \cite{Dingle:1953}, Julius and Hull \cite{Hull:1956}, Singh \cite{Singh:1961}, and Laughlin et al \cite{Laughlin:2002}. However, the arguments used to derive these results can not be considered complete proofs. As far as we are aware, Corollary \ref{cor:hydrogenshift} provides the first rigorous proof of this result.

Formula \eqref{eq:frobsemiclassicalho} for the isotropic harmonic oscillator appears also in \cite{Hull:1956,Singh:1961}. For the linear harmonic oscillator, \eqref{eq:semiclassicalho} was given by Singh \cite{Singh:1959}, and also in \cite{Hull:1956}. Again, these results are not accompanied by rigorous proof. For large quantum numbers (as opposed to the low-lying states considered here), the same formula was also derived by Auluck and Kothari \cite{Auluck:1945} modulo an incorrect factor of $\frac{1}{2}$.

\begin{rem} The aforementioned works give asymptotic formulas as the radius of confinement tends to infinity. By the scaling properties of the linear harmonic oscillator, isotropic harmonic oscillator, and hydrogen atom, these are equivalent to confinement in a box of \emph{fixed} size in the \emph{semiclassical} limit, hence our results apply. For more general potentials in the semiclassical limit (confined to a box of fixed size), Theorems \ref{theo:maintheo1}, \ref{theo:maintheo2} appear to be new. 
\end{rem}

\subsection{Idea of proof} \label{subsect:idea}
Let us briefly describe the strategy used to prove Theorem \ref{theo:maintheo1}. Since $P(h)$ is well approximated by the harmonic oscillator $\widetilde{P}(h)$ near $x=0$, if $u^0$ is an $m$'th eigenvector of $P(h)$ with eigenvalue $\lambda^0$ it is reasonable to expect that \begin{equation} \label{eq:parity}
u^0(0) \neq 0 \text{ if $m$ is even}; \quad (u^0)'(0) \neq 0 \text{ if $m$ is odd},
\end{equation}
since this holds for the eigenvectors of $\widetilde{P}(h)$. The same observation also holds for $P_\Omega(h)$: if $u^\Omega$ is an $m$'th eigenvector of $P_\Omega(h)$ with eigenvalue $\lambda^\Omega$, then $u^\Omega$ should also satisfy \eqref{eq:parity}. It will follow from the WKB construction in Proposition \ref{prop:WKB} that both of these expectations are indeed true.

Fix any $m$'th eigenfunction of $P(h)$ which is polynomially bounded in $h$, namely $\| u^0 \|_{L^2(\mathbb{R})} = \mathcal{O}(h^{-N})$ for some $N>0$. Depending on the parity of $m$, define $u_{\lambda,\beta}$ as the unique nonzero solution to the equation
\[
- h^2 u_{\lambda,\beta}'' + V u_{\lambda,\beta} = \lambda u_{\lambda,\beta},
\]
subject to the initial conditions
\begin{equation*}
u_{\lambda,\beta}(0) = \begin{cases} u^0(0) \text{ if $m$ is even}, \\ \beta \text{ if $m$ is odd}, \end{cases} \quad
u'_{\lambda,\beta}(0) = \begin{cases} \beta \text{ if $m$ is even}, \\ (u^0)'(0) \text{ if $m$ is odd}. \end{cases} 
\end{equation*}
If $\lambda^0$ is the $m$'th eigenvalue of $P(h)$, then of course
\begin{equation} \label{eq:betadefinition}
\text{there exists $\beta^0$ such that } u_{\lambda^0,\beta^0} = u^0.
\end{equation}
Keeping in mind the dependence on $m \geq 0$ and a choice of $u^0$, define $G_\pm(\lambda,\beta) := u_{\lambda,\beta}(r_\pm)$, and then set 
\begin{equation} \label{eq:textbfG}
\mathbf{G}(\lambda,\beta) := \begin{bmatrix} G_+(\lambda,\beta) \\ G_-(\lambda,\beta) \end{bmatrix}.
\end{equation}
The equation $\mathbf{G}(\lambda,\beta) = (0,0)$ is solved by showing that the fixed point iteration
\begin{equation} \label{eq:iteration}
(\lambda^{i+1},\beta^{i+1}) = \mathbf{F}(\lambda^{i},\beta^i) := (\lambda^{i},\beta^{i}) - D\mathbf{G}(\lambda^{0},\beta^{0})^{-1} \mathbf{G}(\lambda^{i},\beta^{i})
\end{equation}
converges to some $(\lambda^\star,\beta^\star)$. We show that $\lambda^\star = \lambda^\Omega$ and then find an asymptotic expansion for $\lambda^\Omega - \lambda^0$.

The same strategy applies to $Q_\Lambda(h)$. Given $\lambda$ and  $\alpha$, there is a unique solution $u_{\lambda}$ to the equation
\begin{equation*}
- h^2 u_{\lambda}'' + h^2(\nu^2 -1/4)x^{-2} u_\lambda + W u_{\lambda} = \lambda u_{\lambda}
\end{equation*}
satisfying $u(x) \sim \alpha x^{1/2+\nu}$ as $x\rightarrow 0$. Fix an $m$'th eigenvector $u^0$ of $Q(\nu;h)$ with eigenvalue $\lambda^0$ satisfying $\| u^0 \|_{L^2((0,\infty))} = \mathcal{O}(h^{-N})$ for some $N>0$, and set 
\[
\alpha = \lim_{x\rightarrow 0} x^{-1/2-\nu}u^0(x).
\] Notice that  $u_{\lambda^0} = u^0$. Define $G(\lambda) = u_\lambda(L)$; this equation is solved by the fixed point iteration
\begin{equation} \label{eq:frobiteration}
\lambda^{i+1} = F(\lambda^i) := \lambda^i - G'(\lambda^0)^{-1} G(\lambda^i).
\end{equation}
Again we show that there exists $\lambda^\star$ such that $\lambda^{i} \rightarrow \lambda^\star$, and moreover that $\lambda^\star = \lambda^\Lambda$, where $\lambda^\Lambda$ is the $m$'th eigenvalue of $Q_\Lambda(\nu;h)$.

\section{Proof of Theorem \ref{theo:maintheo1}}

First observe that by a scaling argument it may be assumed that $V''(0) = 2$: it suffices to replace $V(x)$ with $\widetilde{V}(x) = V(\sqrt{2/V''(0)} \,x)$ and define a new semiclassical parameter $\tilde{h} = \sqrt{V''(0)/2}\,h$. Then the original eigenvalue problem is equivalent to
\[
\left( \tilde{h}^2D_x^2 + \widetilde{V} -E \right) u = 0
\]
with Dirichlet boundary conditions imposed on the boundary of 
\[
\widetilde{\Omega} = (\sqrt{V''(0)/2}\,r_-, \sqrt{V''(0)/2}\,r_+),
\]
where now $\widetilde{V}''(0) = 2$.

Fix an integer $m \geq 0$. Let $\lambda^0$ denote the $m$'th eigenvalue of $P(h)$ and $\lambda^\Omega$ the $m$'th eigenvalue of $P_\Omega(h)$. Let $\beta^0$ be given by \eqref{eq:betadefinition}. As explained in \ref{subsect:idea}, we show that the iterates of $\textbf{F}$ (see \eqref{eq:iteration}) starting with the initial guess $(\lambda^0,\beta^0)$ converge.

\subsection{WKB construction for $P(h)$}

We need to fix a normalization for the eigenfunction $u^0$ of $P(h)$ and then find a tractable approximation to $u^0$. This comes from the WKB construction at a nondegenerate potential minimum. For $P(h)$, this now-standard result is discussed \cite[Chap. 3]{Dimassi:1999}; the points \eqref{itm:hermite}, \eqref{itm:hermitenorm}, \eqref{itm:a0} in Proposition \ref{prop:WKB} below are particular to one dimension, and do not appear explicitly in \cite[Chap. 3]{Dimassi:1999}. Since the complete proof of a very similar result is given in Proposition \ref{prop:frobWKB} of Section \ref{subsect:frobWKB} below, the proof is not indicated for Proposition \ref{prop:WKB}; the interested reader may then complete the proof of Proposition \ref{prop:WKB} by the same methods used to establish Proposition \ref{prop:frobWKB}.

\begin{prop} \label{prop:WKB}
Fix an integer $m \geq 0$ and $\Omega' \supset \Omega$ a bounded open interval. Define $\phi \in C^\infty(\Omega')$ by 
\[
\phi(x) = \sgn x \int_0^x \sqrt{V}(t) dt.
\]
There exists $a_{j}(x) \in C^\infty(\Omega'), \ j \in \mathbb{N}_{\geq 0}$ with $a_0(x) = x^m + \mathcal O(x^{m+1})$, and $a(x,h) \in C^\infty(\Omega')$ with $a(x,h) \sim \sum_{j \geq 0} h^j a_j(x)$, satisfying the following properties.
\begin{enumerate} \itemsep4pt
\item For each compact $K \subset \Omega'$,
\[
\left( P(h) - \lambda^0 \right ) a e^{-\phi/h} = \mathcal{O}_K(h^\infty)e^{-\phi/h}.
\]

\item \label{itm:hermite} There exists $b_j(x)\in C^\infty(\Omega')$ for $0 \leq 2j \leq m$, such that 
\[
\sum_{0\leq 2j \leq m} h^j a_j(x) = 2^{-m}h^{m/2}H_m(h^{-1/2}x) + \sum_{0 \leq 2j \leq m} h^j x^{m-2j+1} b_j(x),
\] 
where $H_m(y)$ is the Hermite polynomial of degree $m$.

\item \label{itm:hermitenorm}
Define $N(h) = \| h^{-1/4} h^{-m/2} a e^{-\phi/h} \|_{L^2(\Omega)}$. Then $N(h)$ admits an asymptotic expansion
\[
N(h) \sim  \sum_{j = 0}^\infty{N_j h^j}, \quad N_0^2 = 2^{-m} m! \sqrt{\pi}.
\]

\item \label{itm:a0} Explicitly,
\[
a_0(x) = \lim_{\varepsilon \rightarrow 0}\, (\varepsilon \sgn x)^m \exp{\left(\int_{\varepsilon\sgn x}^{x} \frac{2m + 1 - \phi''(t)}{2\phi'(t)} dt\right)} = x^m A_0(x),
\]
for some $A_0 \in C^\infty(\Omega')$ with $A_0(x) > 0$.

\item \label{itm:close} Associated with the WKB approximation $h^{-m/2}ae^{-\phi/h}$ is an eigenvector $u^0$ of $P(h)$ satisfying
\[
h^{-m/2}a(x,h)e^{-\phi(x)/h} - u^0(x)  =\mathcal{O}(h^\infty) e^{-\phi(x)/h}, \quad x \in K,
\]
for each compact $K \subset \Omega'$.
\end{enumerate}

\end{prop}

\noindent Choose $u^0$ satisfying \eqref{itm:close} of Proposition \ref{prop:WKB}. Thus
\[
u^0(x) = \left( h^{-m/2}a(x,h) + \delta(x,h) \right) e^{-\phi(x)/h},
\]
where $\delta(x,h) = \mathcal O(h^\infty)$ uniformly on $\Omega$. Furthermore, \eqref{itm:hermite} of Proposition \ref{prop:WKB} verifies the claim made in \eqref{eq:parity} about the values of $u^0, (u^0)'$ at $x=0$ depending on the parity of $m$. Recall that if $\lambda^0$ is the eigenvalue associated to $u^0$, then there exists a unique $\beta^0$ such that
\[
u^0 = u_{\lambda^0,\beta^0},
\]
where $u_{\lambda,\beta}$ is defined as in Section \ref{subsect:idea}. Given one of the subscripts $\alpha \in \{ \lambda, \beta \}$ write
\[
\partial_{\alpha} u^0(x) := \partial_\alpha u_{\lambda,\beta}(x)|_{\lambda = \lambda^0,\beta = \beta^0},
\]
noting that $u_{\lambda,\beta}$ is smooth in the parameters $(\lambda,\beta)$ by standard results from ordinary differential equations.

\subsection{Variation of parameters I} \label{subsect:variationofparam1}

The first task is to compute 
\[
D G_\pm(\lambda^0,\beta^0) = \begin{bmatrix}
\partial_\lambda u^0(r_\pm), & \partial_\beta u^0(r_\pm) \end{bmatrix},
\]
For this we use the variation of parameters formula: suppose that $v^0$ is a complementary solution to the equation $(P(h) - \lambda^0)v^0 = 0$ satisfying $\mathcal{W}(u^0,v^0) = 1$ (here $\mathcal{W}$ denotes the Wronskian). 
Then 
\begin{equation} \label{eq:variationofparameters}
\partial_{\alpha} u^0 (x) =  \mathcal{W}(\partial_\alpha u^0,v^0)(x)u^0(x) - \mathcal{W}(\partial_\alpha u^0,u^0)(x) v^0(x), \quad  \alpha \in \{ \lambda,\beta \}.
\end{equation}
To define the complementary solution $v^0$, first we need a positivity result.

\begin{lem} \label{lem:positivity}
There exists $M > 0$ such that $|u^0(x)| > 0$ for $x\in \Omega \setminus (-Mh^{1/2},Mh^{1/2})$.
\end{lem}
\begin{proof}
Recall that
\[
u^0(x) = \left( h^{-m/2}a(x,h) + \delta(x,h) \right) e^{-\phi(x)/h} ,
\]
where $\delta(x,h) = \mathcal{O}(h^\infty)$ uniformly on $\Omega$. For each $\varepsilon>0$ there exists $C_\varepsilon>0$ such that $|a_0(x)| > 2/C_\varepsilon$ for $|x| > \varepsilon$; this follows from \eqref{itm:a0} of Proposition \ref{prop:WKB}. Therefore
\begin{equation} \label{eq:upositive1}
|u^0(x)| > C_\varepsilon^{-1} h^{-m/2} e^{-\phi(x)/h}, \quad  |x| > \varepsilon.
\end{equation} 
On the other hand, write the Hermite polynomial $H_m$ as
\[
H_m(y) = d_m y^m + d_{m-2} y^{m-2} + \cdots + d_0, \quad d_m = 2^m > 0,
\]
and choose $M> 0$ and $C_M$ such that $|d_m y^m| > 2 C_M$ and $|d_{m-2k} y^{m-2k}| < \frac{1}{2 m} | d_m y^m |$ for $|y| > M$ and $0 < 2k \leq m$. Referring to \eqref{itm:hermite} of Proposition \ref{prop:WKB} for the definition of $b_k$, choose $\varepsilon > 0$ such that 
\[
|2^m x^{m-2k+1} b_{k}(x)| \leq |d_{m-2k} x^{m-2k}|/2, \quad | x | \leq \varepsilon
\]
for $0 \leq 2k \leq m$. It easily follows from this that  
\begin{equation} \label{eq:upositive2}
|u^0(x)| \geq C'_M e^{-\phi(x)/h}, \quad x\in [-\varepsilon,\varepsilon] \setminus [-Mh^{1/2},Mh^{1/2}]
\end{equation}
for some $C'_M > 0$. In particular, combining \eqref{eq:upositive1}, \eqref{eq:upositive2} shows that $|u^0(x)| > 0$ for $u\in \Omega \setminus [-Mh^{1/2},Mh^{1/2}]$.
\end{proof}
The complementary solution $v^0$ is defined by the standard ansatz: choose $M > 0$ such that the conclusion of Lemma \ref{lem:positivity} holds, and define 
\begin{equation} \label{eq:secondsolution}
v^0(x) = u^0(x) \int_{\pm M h^{1/2}}^x u^0(t)^{-2} \,dt, \quad \pm x \geq Mh^{1/2}.
\end{equation}
Then $v^0$ solves 
\[
\begin{cases}
(P(h)-\lambda^0 )v^0 = 0,\\ v^0(\pm Mh^{1/2}) = 0,\quad  (v^0)'(\pm Mh^{1/2}) = u^0(\pm M h^{1/2})^{-1}.
\end{cases}
\]
Then next lemma provides an asymptotic expansion for $v^0(r_\pm)$.
\begin{lem} \label{lem:secondsolution}
If $v^0$ is defined by \eqref{eq:secondsolution}, then
\begin{equation} \label{eq:vrpm}
v^0(r_\pm) = h^{\frac{m}{2}+1} f_\pm(h) e^{\phi(r_\pm)/h},
\end{equation}
where $f_\pm(h)$ has an asymptotic expansion
\begin{equation} \label{eq:frpm}
f_{\pm}(h) \sim \sum_{j=0}^\infty f^\pm_{j} h^j, \quad f^\pm_{0} = \frac{\pm1}{2\sqrt{V(r_\pm)}}a_0(r_\pm)^{-1}.
\end{equation}
\end{lem}
\begin{proof}

To prove \eqref{eq:vrpm}, \eqref{eq:frpm} notice that $\phi$ is strictly convex so using  \eqref{eq:upositive2} we may write 
\[
\int_{\pm Mh^{1/2}}^{r_\pm} u^0(t)^{-2}\, dt = \int_{\pm\varepsilon}^{r_\pm} h^{m} e^{2\phi(t)/h}\left( a(t,h)^{-2} + \mathcal{O}(h^\infty)\right) dt + \mathcal{O}\left(\exp(2\phi(\pm\epsilon)/h)\right).
\]  
The phase $2\phi(t)$ achieves its maximum at $r_\pm$, so evaluating the integral by Laplace's method \cite[Theorem 8.2]{Olver:1997} we get \eqref{eq:vrpm}, where $f_\pm(h)$ satisfy \eqref{eq:frpm}.
\end{proof}

To calculate the Wronskians $\mathcal{W}(\partial_\alpha u^0,u^0)(r_\pm)$  and $\mathcal{W}(\partial_\alpha u^0,v^0)(r_\pm)$ for  $\alpha \in \{\lambda,\beta\}$, use that $\partial_\lambda u^0$ solves
\begin{equation} \label{eq:lambdaderiv}
\begin{cases}
(P(h)-\lambda)\partial_\lambda u_{\lambda,\beta} = u_{\lambda,\beta},\\ \partial_\lambda u_{\lambda,\beta}(0) = 0, \quad  \partial_\lambda u'_{\lambda,\beta}(0) = 0,
\end{cases}
\end{equation}
and $\partial_\beta u^0$ solves
\begin{equation} \label{eq:betaderiv}
(P(h)-\lambda)\partial_\beta u_{\lambda,\beta} = 0
\end{equation}
subject to the initial conditions
\[
 \partial_\beta u_{\lambda,\beta}(0) = \begin{cases} 0 \text{ if $m$ is even},\\ 1 \text{ if $m$ is odd}, \end{cases}\ \partial_\beta u'_{\lambda,\beta}(0) = \begin{cases} 1 \text{ if $m$ is even}, \\ 0 \text{ if $m$ is odd}. \end{cases}
\]
First we need to control how rapidly solutions to \eqref{eq:lambdaderiv}, \eqref{eq:betaderiv} can grow.

\begin{lem} \label{lem:growth}
Let $K$ be a compact subinterval of $\mathbb{R}$. Suppose $u\in C^2(K)$ solves
\[
(h^2 D_x^2 + V(x) - \lambda)u = f
\]
on $K$, where $0 \leq \lambda \leq C_0 h$ for some $C_0 > 0$ and $f\in L^2(K)$. Then there exists $C>0$ depending on $K$ and $C_0$ such that
\begin{multline*}
e^{-\phi(x_1)/h}\left(h^{1/2}|u(x_1)| + |hu'(x_1) + \sqrt{V(x)}u(x_1)|\right) \\ \leq C  e^{-\phi(x_0)/h}\left(h^{1/2}|u(x_0)| + |hu'(x_0) + \sqrt{V}(x_0)u(x_0)| + h^{-1}\|e^{-\phi/h}f\|_{L^2(x_0,x_1)}\right)
\end{multline*}
for $x_0, x_1\in K$. 
\end{lem}
\begin{proof}
Only the case $x_0=0$ is treated, but it will be clear from the proof that this is not necessary. It is also assumed that $x \geq 0$; the case $x \leq 0$ is handled identically. Write 
\[
u = e^{\phi(x)/h}v; \quad f = e^{\phi(x)/h}g
\]
and set $\lambda = hE$ with $0 \leq E \leq C_0$, so that
\[
e^{-\phi(x)/h} P(h) e^{\phi(x)/h} v = \left( h^2 D_x^2 - 2 \phi'(x) h \partial_x - h\phi''(x) - hE\right)w = g.
\]
Begin by choosing $\varepsilon > 0$ such that $\varepsilon \leq \phi''(x) \leq C_0$ for $x\in [0,\varepsilon]$ --- this is possible since the minimum of $V$ at $x=0$ is nondegenerate. Set $A(x) = h(\phi''(x)+E)$, so that 
\begin{equation} \label{eq:A(X)}
h\varepsilon \leq A(x) \leq 2C_0 h, \quad x\in [0,\varepsilon].
\end{equation}
Furthermore, there exists $C_1>0$ such that $A'(x) = h\phi'''(x) \leq C_1h$. Calculate
\begin{align*}
\tfrac{1}{2} \partial_x \left( A(x) |w(x)|^2 + h^2|w'(x)|^2 \right) &=  \Re \left (\left[h^2 w''(x)+A(x)w(x)\right]\overline{w}'(x) \right) + A'(x)|w(x)|^2 \\
&= -2 h \phi'(x) |w'(x)|^2 + A'(x)|w(x)|^2 - \Re\left(g(x) \overline{w}'(x)\right)
\end{align*}
Now $A'(x) \leq C_2 A(x)$ where $C_2$ is independent of $h$. Using 
\[
\Re \left(g(x) \overline{w}'(x)\right) \leq h^2|w'(x)|^2 + h^{-2}|g(x)|^2, 
\]
we obtain
\[
\tfrac{1}{2} \partial_x \left( A(x) |w(x)|^2 + h^2|w'(x)|^2 \right) \leq C_3 \left(A(x) |w(x)|^2 + h^2|w'(x)|^2 + h^{-2}|g(x)|^2 \right).
\]
Applying Gronwalls inequality, 
\[
A(y) |w(y)|^2 + h^2|w'(y)|^2 \leq C_4 \left( A(0)|w(0)|^2 + h^2|w'(0)|^2 + h^{-2} \| g(x) \|_{L^2(0,\epsilon)} \right).
\]
Finally, use \eqref{eq:A(X)} to bound $A(y)$ from below and $A(0)$ from above.

Now consider the interval $[\varepsilon, x_1]$. On this interval $\phi'(x) > \delta$ for some $\delta > 0$. This time calculate
\begin{align*}
&\frac{1}{2} \partial_x( h |w(x)|^2 + h^2|w'(x)|^2) =  \Re \left(\left[h^2 w''(x)+h w(x)\right]\overline{w}'(x)\right)  \\
&= -2 h \phi'(x) |w'(x)|^2 + h(1 - \phi''(x) - E)w(x) \overline{w}'(x) - \Re\left(g(x) \overline{w}'(x)\right).
\end{align*}
For any $R>0$ and $x\in [\varepsilon,x_1]$,
\[
h(1 - \phi''(x) - E) w(x) \overline{w}'(x) \leq h (R|w|^2 + |w'|^2/R).
\]
By choosing $R>0$ sufficiently large, it follows that $-2\phi'(x) + 1/R < 0$ uniformly on $K$. It remains to apply Gronwall's inequality once more on the interval $[\varepsilon,x_1]$.

\end{proof}

\noindent The next step is to calculate $\partial_\lambda u^0(r_\pm)$ and $\partial_\beta u^0(r_\pm)$.

\begin{lem} \label{lem:wronskians1}
With $N(h)$ given by \eqref{itm:hermitenorm} of Proposition \ref{prop:WKB},
\begin{enumerate} \itemsep4pt
\item $\mathcal W(\partial_\lambda u^0, u^0)(r_\pm) = \pm 2^{-1} h^{-3/2 + m} N(h)^2$,
\item $\mathcal W(\partial_\lambda u^0, v^0)(r_\pm) = \mathcal{O}(h^{-K})$ for some $K>0$.
\end{enumerate}

Consequently,
\begin{equation} \label{eq:lambdau0rpm}
\partial_\lambda u^0(r_\pm) = \mp\frac{h^{-3/2 + m}N(h)^2}{2} v(r_\pm) + \mathcal{O}\left(h^{-K}e^{-\phi(r_\pm)/h}\right).
\end{equation}
\end{lem}
\begin{proof}
$(1)$ Integrate the Wronskian identity
\[
h^2 \partial_x \mathcal W(\partial_\lambda u^0, u^0)(x) = u^0(x)^2
\] 
to obtain
\[
\mathcal W(\partial_\lambda u^0, u^0)(r_{\pm}) = h^{-2} \int_{0}^{r_{\pm}} u^0(t)^2 dt,
\]
using \eqref{eq:lambdaderiv} to compute the initital condition $\mathcal {W}(\partial_\lambda u^0, u^0)(0) = 0$. Now replace $u^0(x)$ with $(a(x,h) + \delta(x,h))e^{-\phi(x)/h}$. As in the proof of Lemma \ref{lem:secondsolution}, modulo an exponentially small relative error change the domain of integration to $[0,\varepsilon]$ and $[-\varepsilon,0]$, and then calculate (half) the $L^2$ norm from Proposition \ref{prop:WKB}.

$(2)$ Similarly,
\[
\mathcal W(\partial_\lambda u^0, v^0)(r_{\pm}) = \mathcal W(\partial_\lambda u^0, v^0)(\pm Mh^{1/2}) + h^{-2} \int_{\pm Mh^{1/2}}^{r_{\pm}} v^0(t) u^0(t) dt.
\]
From Lemma \ref{lem:secondsolution}, $u^0(x) v^0(x) = \mathcal O(h^{-K})$ for $x\in \Omega \setminus [-Mh^{1/2},Mh^{1/2}]$. Furthermore, applying Lemma \ref{lem:growth} to the equation \eqref{eq:lambdaderiv} satisfied by $\partial_\lambda u^0$, it follows that $\partial_\lambda u^0 (\pm Mh^{1/2}) = \mathcal O(h^{-K})$. Combining this with the initial conditions satisfied by $v^0$ according to \ref{lem:secondsolution}, $\mathcal W(\partial_\lambda u^0, v^0)(\pm Mh^{1/2}) = \mathcal O(h^{-K})$.

Now \eqref{eq:lambdau0rpm} follows from \eqref{eq:variationofparameters}.

\end{proof}

\begin{lem} \label{lem:wronskians2}
The Wronskians $W(\partial_\beta u^0,u^0)$ and $W(\partial_\beta u^0,v^0)$ are constant functions satisfying the following.
\begin{enumerate}
\item If $m$ is even, then 
\[
\mathcal W(\partial_\beta u^0,u^0) = u^0(0),
\]
while if $m$ is odd, then
\[
\mathcal W(\partial_\beta u^0,u^0) = -(u^0)'(0).
\]
\item $\mathcal W(\partial_\beta u^0,v^0) = O(h^{-K})$ for some $K>0$.
\end{enumerate}

Consequently,
\begin{equation} \label{eq:betau0rpm}
\partial_\beta u^0(r_\pm) = -\mathcal W(\partial_\beta u^0,u^0) v(r_\pm) + \mathcal{O}\left( h^{-K} e^{-\phi(r_\pm)/h} \right). 
\end{equation}
\end{lem} 
\begin{proof}
From \eqref{eq:betaderiv}, $\partial_\beta u^0$ solves the homogeneous equation, and hence each of the Wronskians is constant.

$(1)$ Calculate $\mathcal W(\partial_\beta u^0,u^0)$ at $x=0$ using the initial conditions given by \eqref{eq:betaderiv}.

$(2)$ Apply Lemma \ref{lem:growth} evaluated at $\pm Mh^{1/2}$ to get that $\mathcal W(\partial_\beta u^0,v^0) = \mathcal O(h^{-K})$.

Again \eqref{eq:betau0rpm} follows by applying \eqref{eq:variationofparameters}.

\end{proof}

\begin{rem}Note that $u^0(0)$ and $(u^0)'(0)$ are both polynomially bounded in $h$ as well, and hence the (absolute) error in  \eqref{eq:betau0rpm} is exponentially small compared to $\mathcal W(\partial_\beta u^0,u^0) v(r_\pm)$. 
\end{rem}

Combining \eqref{eq:lambdau0rpm} with \eqref{eq:betau0rpm}, yields the formula 
\begin{multline} \label{eq:DGpm}
DG_\pm(\lambda^0,\beta^0) =
\begin{bmatrix}
\mp2^{-1} h^{-3/2+m} N(h)^2, &  -\mathcal{W}(\partial_\beta u^0,u^0)\end{bmatrix}v(r_\pm) \\ + \begin{bmatrix}
\mathcal{O}\left( h^{-K} e^{-\phi(r_\pm)/h} \right),& \mathcal{O}\left( h^{-K} e^{-\phi(r_\pm)/h} \right) \end{bmatrix}.
\end{multline}
From \eqref{eq:DGpm} it is easy to calculate $D\mathbf{G}(\lambda^0,\beta^0)^{-1}$: define
\begin{align*}
\mathbf{V}_\lambda &= h^{3/2-m}N(h)^{-2} \begin{bmatrix} v(r_-)^{-1}, &-v(r_+)^{-1} \end{bmatrix}; \\
\mathbf{V}_\beta &= \frac{1}{2}\mathcal W(\partial_\beta u^0,u^0)^{-1} \begin{bmatrix} -v(r_-)^{-1}, & -v(r_+)^{-1} \end{bmatrix}.
\end{align*}
Since $v(r_\pm)^{-1} = \mathcal{O}(h^{-K} e^{-\phi(r_\pm)/h})$, it follows that
\[
D\mathbf{G}(\lambda^0,\beta^0)^{-1} = \begin{bmatrix}
\mathbf{V}_\lambda \\
\mathbf{V}_\beta 
\end{bmatrix} +  \begin{bmatrix}
\mathcal{O}(h^{-K}e^{-3\phi(r_-)/h}) & \mathcal{O}(h^{-K}e^{-3\phi(r_+)/h}) \\
\mathcal{O}(h^{-K}e^{-3\phi(r_-)/h}) & \mathcal{O}(h^{-K}e^{-3\phi(r_+)/h})
\end{bmatrix}.
\]

The following proposition summarizes the different pieces of information needed to prove Theorem \ref{theo:maintheo1}.

\begin{prop} \label{lem:mainlemma1}
Fix an integer $m\geq 0$.
\begin{enumerate}
\item With $a(x,h)$ given by Proposition \ref{prop:WKB},
\[
G_\pm(\lambda^0,\beta^0) = h^{-m/2} a(r_\pm,h)e^{-\phi(r_\pm)/h} + \mathcal{O}\left(h^\infty e^{-\phi(r_\pm)/h}\right).
\]

\item The derivative $D\mathbf{G}(\lambda^0,\beta^0)$ is invertible, and 
\[
D\mathbf{G}(\lambda^0,\beta^0)^{-1} =
\begin{bmatrix}
h^{\frac{1-m}{2}}  p_-(h) \, \exp(-\phi(r_-)/h)&  -h^{\frac{1-m}{2}}  p_+(h) \, \exp(-\phi(r_+)/h) \\
q_-(h) \, \exp(-\phi(r_-)/h)  &  q_+(h) \, \exp(-\phi(r_+)/h) 
\end{bmatrix}.
\]
Here $q_\pm(h) = \mathcal{O}(h^{-K})$, while $p_\pm(h)$ admits the same asymptotic expansion as $N(h)^{-2} f_\pm(h)^{-1}$ so that 
\[
p_\pm(h) \sim \sum_{i=0}^{\infty} p_j^\pm h^j, \quad p_0^\pm = N_0^{-2} (f_0^\pm)^{-1} = \pm \frac{2^{m+1}}{ m! \sqrt{\pi}}\sqrt{V(r_\pm)}\,a_0(r_\pm).
\]

\item \label{itm:D2G} Given $C_0,C_1 > 0$, suppose that $0 \leq \lambda \leq C_0h$ and  $|\beta| < C_1$. Then 
\[
|D^2 G_\pm(\lambda,\beta)| = \mathcal{O}\left(h^{-K}e^{\phi(r_\pm)/h}\right).
\]

\end{enumerate}
\end{prop}

\begin{proof}
The only part that hasn't already been established is \eqref{itm:D2G}. For this, use the equations

\begin{gather*}
\begin{cases} (P(h)-\lambda)\partial^2_\lambda u_{\lambda,\beta} = 2 \partial_\lambda u_{\lambda,\beta},\\ \partial_\lambda^2 u_{\lambda,\beta}(0) = \partial^2_\lambda u'_{\lambda,\beta}(0) = 0, \end{cases} \quad \begin{cases} (P(h)-\lambda)\partial^2_\beta u_{\lambda,\beta} = 0,\\ \partial^2_\beta u_{\lambda,\beta}(0) = \partial^2_\lambda u'_{\lambda,\beta}(0) = 0, \end{cases} \\ \begin{cases} (P(h)-\lambda)\partial_{\lambda,\beta} u_{\lambda,\beta} = \partial_\beta u_{\lambda,\beta}, \\ \partial_{\lambda,\beta} u_{\lambda,\beta}(0) = \partial_{\lambda,\beta} u'_{\lambda,\beta}(0) = 0, \end{cases}
\end{gather*}
%
%
and then apply Lemma \ref{lem:growth}.
\end{proof}

Using Lemma \ref{lem:mainlemma1}, it is now straightforward to prove Theorem \ref{theo:maintheo1}. The crux of the argument lies in showing that $\mathbf{F}$, defined by \eqref{eq:iteration}, is a contraction mapping in a suitable ($h$-dependent) neighborhood of $(\lambda^0,\beta^0)$.

\begin{proof} [Proof of Theorem \ref{theo:maintheo1}]
Write 
\[
\mathbf{F}(\lambda,\beta) = \begin{bmatrix}
\mathbf{F}_\lambda(\lambda,\beta) \\ \mathbf{F}_\beta(\lambda,\beta)
\end{bmatrix},
\]
where $\mathbf{F}_{\alpha} : \mathbb{R}^2 \rightarrow \mathbb{R}, \, \alpha \in \{ \lambda,\beta \}$.

First we show that there exists $0 \leq \gamma \ll 1$ and $L>0$ such that $|D\mathbf{F}(\lambda,\beta)| < \gamma$ for $|\lambda- \lambda^0| + |\beta-\beta^0| \leq  h^L$. We have
\[
D\mathbf{F}(\lambda,\beta) = I - D\mathbf{G}(\lambda^0,\beta^0)^{-1} D\mathbf{G}(\lambda,\beta).
\]
First, note that $D\mathbf{F}(\lambda^0,\beta^0) = 0$. On the other hand, it follows from Lemma \ref{lem:mainlemma1} that
\begin{equation} \label{eq:Fsecondderivative}
|D^2 \textbf{F}_\alpha (\lambda,\beta)| = \mathcal{O}(h^{-K}), \quad \alpha \in \{ \lambda, \beta \}
\end{equation}
for some $K>0$, hence the result follows by taking $L \gg K$ and applying Taylor's theorem. Furthermore, this also shows that 
\[
\textbf{F}: \{|\lambda- \lambda^0| + |\beta-\beta^0| \leq  h^L\} \rightarrow \{ |\lambda- \lambda^0| + |\beta-\beta^0| \leq  h^L \}.
\]
By the contraction mapping principle, the sequence of iterates $(\lambda^i,\beta^i)$, given recursively by $(\lambda^i,\beta^i) = \textbf{F}(\lambda^{i-1},\beta^{i-1})$, converges to a unique root 
\[
(\lambda^\star,\beta^\star)\in \{|\lambda- \lambda^0| + |\beta-\beta^0| \leq  h^L\}.
\] 
If $L$ is larger than one it follows from \eqref{eq:harmonicapproximation} that $\lambda^\star = \lambda^\Omega$. 

Therefore we may write
\[
(\lambda^\Omega,\beta^\Omega) - (\lambda^0,\beta^0) = (\lambda^1,\beta^1) - (\lambda^0,\beta^0) + \sum_{j=1}^\infty \left( (\lambda^{j+1},\beta^{j+1}) - (\lambda^{j},\beta^{j}) \right),  
\]
and
\[
\sum_{j=1}^\infty  |(\lambda^{j+1},\beta^{j+1}) - (\lambda^{j},\beta^{j}) | \leq \frac{1}{1-\gamma}\left( |\lambda^2-\lambda^1| + |\beta^2-\beta^1| \right).
\]
Now by definition,
\[
(\lambda^2,\beta^2) - (\lambda^1,\beta^1) = \textbf{F}(\lambda^1,\beta^1) - \textbf{F}(\lambda^0,\beta^0).
\]
Taylor expand to second order around $(\lambda^0,\beta^0)$, using that $D\textbf{F}(\lambda^0,\beta^0) = 0$ along with the bound \eqref{eq:Fsecondderivative} to obtain
\[
\textbf{F}(\lambda^1,\beta^1) - \textbf{F}(\lambda^0,\beta^0) = \mathcal{O}\left(h^{-K}\right)\left( |\lambda^1- \lambda^0|^2 + |\beta^1 - \beta^0|^2 \right).
\]
Furthermore,
\[
(\lambda^1,\beta^1) - (\lambda^0,\beta^0) = -D\mathbf{G}(\lambda^0,\beta^0)^{-1} \mathbf{G}(\lambda^0,\beta^0) = \mathcal{O}\left(h^{-K}\right)\sum_\pm e^{-2\phi(r_\pm)/h},
\]
which shows that $|\lambda^2-\lambda^1| + |\beta^2 - \beta^1| = \mathcal{O}\left(h^{-K}\right)\sum_{\pm} e^{-4\phi(r_\pm)/h}$. Consequently
\[
\lambda^\Omega - \lambda^0 = h^{\frac{1}{2}-m}\sum_\pm e^{-2\phi(r_\pm)}s_\pm(h),
\]
where $s_\pm(h)\sim \sum_{j=0}^\infty s^\pm_j\, h^j$ admits the same asymptotic expansion as 
\[
\pm N(h)^{-2} f_\pm(h)^{-2} a(r_\pm,h),
\] 
and hence
\[
s^\pm_0 = \frac{2^{m+1}}{m! \, \pi^{\frac12}} \sqrt{V(r_\pm)}\, a_0(r_\pm)^2. 
\]

\end{proof}

\section{Proof of Theorem \ref{theo:maintheo2}}
The proof of Theorem \ref{theo:maintheo2} follows the same steps as that of Theorem \ref{theo:maintheo1}. By a rescaling argument it may be assumed that $W''(0) = 2$. Fix an integer $m \geq 0$ and let $\lambda^0$ denote the $m$'th eigenvalue of $Q(\nu;h)$ and $\lambda^\Lambda$ the $m$'th eigenvalue of $Q_\Lambda(\nu;h)$.

\begin{lem} [{{\cite{Bocher:1900}}}] \label{lem:frobmethod}
Let $\nu > 0$. Suppose that $B(x;z)$ is smooth in $(x,z) \in [0,L) \times Z$, where $0 < L \leq \infty$ and $Z \subset \mathbb{R}$ is a connected open interval. Then there exists a solution $u$ to the equation
\[
-h^2 u'' + h^2 (\nu^2 - 1/4)x^{-2}u + B(x;z)u = 0
\]
of the form $u = x^{1/2+\nu}w$, where $w(x;z)$ is smooth in $[0,L) \times Z$ and
\[
w(0) = 1, \quad w'(0) = 0.
\]
\end{lem}
   
Any $H_0^1((0,\infty))$ function of the form given by Lemma \ref{lem:frobmethod} lies in the domain of $Q(\nu;h)$ \cite{Everitt:2007}, so any eigenvector of $Q(\nu;h)$ is also of this form. The same observation holds for the eigenvectors of $Q_{\Lambda}(\nu;h)$.

\subsection{WKB construction for $Q(\nu;h)$} \label{subsect:frobWKB}
We need a WKB construction for the $m$'th eigenvector of $Q(\nu;h)$. Since this result is not standard, a proof is provided.

\begin{prop} \label{prop:frobWKB}
Fix an integer $m \geq 0$ and $\Lambda' \supset \Lambda$ of the form $\Lambda' = [0,L')$. Define $\phi \in C^\infty(\Lambda')$ by 
\[
\phi(x) = \int_0^x \sqrt{W}(t) \,dt.
\]
There exists $a_{j}(x) \in C^\infty(\Lambda'), \ j \in \mathbb{N}_{\geq 0}$ with $a_0(x) = x^{2m} + \mathcal O(x^{2m+2})$, and $a(x,h) \in C^\infty(\Lambda')$ with $a(x,h) \sim \sum_{j \geq 0} h^j a_j(x)$, satisfying the following properties.

\begin{enumerate} \itemsep4pt

\item $a^{(2k+1)}_j(0) = 0$ for $j \geq 0, \, k \geq 0$.

\item For each compact $K \subset \Lambda'$,
\[
\left( Q(\nu;h) - \lambda^0 \right ) \left(x^{1/2 + \nu} a e^{-\phi/h}\right) = \mathcal O_K(h^\infty) x^{1/2 + \nu}e^{-\phi/h}.
\]

\item There exists $b_j(x)\in C^\infty(\Lambda')$ for $0 \leq j \leq m$, such that 
\[
\sum_{0\leq j \leq m} h^j a_j(x) = (-1)^m h^{m} \, m! \, L^{(\nu)}_m(h^{-1} x^2) + \sum_{0 \leq j \leq m} h^j x^{2m-2j+2} b_j(x),
\] 
where $L^{(\nu)}_m(y)$ is the Laguerre polynomial of degree $m$.

\item \label{itm:laguerrenorm}
Define $N(\nu;h) := \| h^{-1/2-\nu/2}h^{-m} x^{1/2 + \nu} a e^{-\phi/h} \|_{L^2(\Lambda)}$. Then $N(\nu;h)$ admits an asymptotic expansion
\[
N(\nu;h) \sim \sum_{j = 0}^\infty N_j(\nu) h^j, \quad N_0(\nu)^2 = \frac{ \Gamma(1+m+\nu) m!}{2}.
\]

\item Explicitly,
\begin{align*}
a_0(x) &= \lim_{\varepsilon \rightarrow 0} \varepsilon^{2m} \exp{\left(\int_\varepsilon^x \frac{ 2(2m + 1 +\nu) - \phi''(t) - (2\nu+1)t^{-1} \phi'(t)}{2\phi'(t)} dt\right)}\\ &= x^{2m} A_0(x),
\end{align*}
for some $A_0 \in C^\infty(\Lambda')$ with $A_0(x) > 0$.

\item \label{itm:frobeigen} Associated with the WKB approximation $h^{-m}x^{1/2+\nu}ae^{-\phi/h}$ is an eigenvector $u^0$ of $Q(\nu;h)$ satisfying
\[
h^{-m}x^{1/2+\nu}a(x,h)e^{-\phi(x)/h} - u^0(x) = \mathcal{O}(h^\infty) e^{-\phi(x)/h}, \quad x \in K
\]
for each compact $K \subset \Lambda'$.
\end{enumerate}
\end{prop}

\begin{proof}

Conjugating $Q(\nu;h)$ by $x^{1/2 + \nu}e^{-\phi(x)/h}$ yields
\begin{multline} \label{eq:frobconjugation}
e^{\phi(x)/h} x^{-\nu - 1/2} (Q(\nu;h)-hE) x^{\nu + 1/2} e^{-\phi(x)/h}   \\
=h^2 \left( D_x^2 - (2\nu+ 1)x^{-1} \partial_x \right) + h \left( 2\phi'(x) \partial_x + \phi''(x) + (2\nu + 1)x^{-1} \phi'(x) - E \right).
\end{multline}

Since the Taylor series of $W$ at $x=0$ contains only even terms and $\phi(0) = 0$, it follows that $x^{-1} \phi'(x)$ is smooth in $x^2$. Define the differential operator
\[
L = 2\phi'(x) \partial_x + \phi''(x) + (2\nu + 1)x^{-1}\phi'(x).
\]
Then plugging in a formal expansion $E \sim \sum_{j=0}^\infty E_j h^j$ and $a(x,h) \sim \sum_{j=0}^{\infty} a_j(x) h^j$ into \eqref{eq:frobconjugation} and equating powers of $h$, we obtain the sequence of transport equations
\begin{align} \label{eq:frobtransport1}
&(L - E_0)a_0 = 0, \\
&(L - E_0)a_j = \left(\partial_x^2 + (2\nu+1)x^{-1}\partial_x\right)a_0 + \sum_{k=1}^j E_k a_{j-k}(x), \quad   j \geq 1. \label{eq:frobtransport2}
\end{align}

Although these equations can be solved by ODE methods, instead we follow \cite[Chap. 3]{Dimassi:1999} and first solve \eqref{eq:frobtransport1}, \eqref{eq:frobtransport2} by formal power series; this approach clarifies the role of Assumption \eqref{itm:fifth}. If $\mathbb{C}[[x^2]]$ denotes the space of formal power series in $x^2$, let $D_l[x]$ denotes the one-dimensional subspace spanned by $x^{2l}$. Acting on $\mathbb{C}[[x^2]]$,
\[
L = L_0 + \sum_{k=1}^\infty L_{k}; \quad L_0 = 2 x \partial_x + 2\nu + 2,
\] 
and $L_k : D_l[x]\rightarrow D_{l+k}[x]$ for each $k\geq 0$. Then $L_0$ acting on $D_{l}[x]$ has eigenvalue $2(2l+1+\nu)$. For each integer $m\geq 0$ we can solve the first transport equation $(L - E_0)\widetilde{a}_0=0$ by setting $E_0 = 2(2m+1+\nu)$ and $\widetilde{a}_0 = x^{2m} + \mathcal O(x^{2m+2})$, and then iteratively determining the higher terms in $\widetilde{a}_0$. On the other hand $\partial_x^2 + (2\nu + 1)x^{-1}\partial_x : D_{l+1}[x] \rightarrow D_{l}[x]$, and it is easy to see that there exists a unique $E_j$ so that \eqref{eq:frobtransport2} admits a solution $\widetilde{a}_j \in \mathbb{C}[[x^2]]$.

By a slight abuse of notation, also write $\widetilde{a}_j$ for any fixed $C^\infty(\Lambda')$ function with the given Taylor series obtained by Borel summation. Let $r_0 = (L - E_0)\widetilde{a_0}$. Then $r_0$ is smooth and $r_0 = \mathcal{O}(|x|^\infty)$; using \cite[Chap. 3, Prop 3.5]{Dimassi:1999}, we can solve $(L - E_0)\hat{a}_0 = r_0$ for a smooth $\hat{a} = \mathcal{O}(|x|^\infty)$, and then define $a_0 = \widetilde{a}_0 - \hat{a}_0$. Similarly, each $\widetilde{a}_j$ can be corrected by a function $\hat{a}_j$ vanishing to infinite order at $x=0$, so that $a_j = \widetilde{a}_j - \hat{a}_j$ solves the given transport equation. By construction $(1)$ holds, and $(2)$ follows from a standard argument using the spectral theorem.

$(3)$ Notice that $a_j(x) = c_j x^{2m - 2j} + \mathcal{O}(|x|^{2m - 2j + 2})$ for $0 \leq j \leq m$, and $c_j$ depends only on $W''(0)$. Calculating the recursion relation satisfied by the $c_j$,
\[
c_0 x^{2m} + c_1 h x^{2m-2} \cdots + c_m h^m = (-1)^m h^{m} \, m! \, L^{(\nu)}_m(h^{-1} x^2). 
\]

$(4)$ The Laguerre polynomials satisfy 
\[
\| h^{-\frac{1+\nu}{2}} x^{1/2+\nu}L_m^{(\nu)}(h^{-1}x^2) e^{-x^2/2h} \|^2_{L^2((0,\infty))} =  \frac{ \Gamma(1+m+\nu) m!}{2}.
\]
The result follows from comparing the Laplace expansion of this integral with that of $ \| h^{-\frac{1+\nu}{2}} x^{1/2+\nu} a(x) e^{-\phi(x)/h} \|_{L^2(\Lambda)}^2$ using $(3)$. A priori, the latter asymptotic expansion is in powers of $h^{1/2}$ but the odd terms vanish since $a^{(2k+1)}_j(0) = 0$.

$(5)$ The equation \eqref{eq:frobtransport1} for $a_0$ can be solved explicitly since it is a first order ordinary differential equation (with a singular point at $x=0$).

$(6)$ This fact relies on Agmon estimates \cite[Chap. 6]{Dimassi:1999}. If $u$ in the domain of $Q(\nu;h)$ satisfies $u(L')=0$ and $\Phi \in C^2([0,L'])$ then the following identity holds,
\begin{multline} \label{eq:agmon}
\left< (h^2D_x^2 + h^2(\nu^2 - 1/4)x^{-2} )(e^{\Phi/h}u),e^{\Phi/h}u \right>_{L^2((0,L'))} \\ + \left< (W - (\Phi')^2)e^{\Phi/h}u,e^{\Phi/h}u \right>_{L^2((0,L'))}
= \Re \left < e^{2\Phi/h} Q(\nu;h) u, u \right >_{L^2((0,L'))}.
\end{multline}

By an approximation procedure this also holds for $\Phi$ which is Lipschitz on $[0,L']$ (so $\Phi'$ exists almost everywhere). Furthermore, if $\nu > 0$, then Hardy's inequality shows that if $v$ is in the domain of $Q(\nu;h)$ and $v(L') = 0$, then $v \in H_0^1((0,L'))$. Furthermore,
\[
\| v \|^2_{H^1((0,L'))} \leq C_\nu \left< (h^2D_x^2 + h^2(\nu^2 - 1/4)x^{-2} )v,v \right>_{L^2((0,L'))}.
\]
Therefore
\begin{multline*}
\| D_x ( e^{\Phi/h} u) \|_{L^2((0,L'))}^2 + \left< (W - (\Phi')^2)e^{\Phi/h}u,e^{\Phi/h}u \right>_{L^2((0,L'))} \\ \leq C \Re \left < e^{2\Phi/h} Q(\nu;h) u, u \right >_{L^2((0,L'))}
\end{multline*}
The proof of \cite[Chap. 6, Theorem A.3]{Dimassi:1999} now goes through identically since it depends only on an appropriate choice of phase $\Phi$. Thus for an appropriately normalized eigenfunction $u^0$ of $Q(\nu;h)$ and $K\subset \Lambda'$ of the form $K= [0,L'']$ with $0 < L'' < L'$,
\[
\| h^{-m} x^{1/2 + \nu}a(x,h) - e^{\phi(x)/h} u^0 \|_{H^1 (K)} \leq C_{K,N} h^{N}.
\]
It then remains to apply the one-dimensional Sobolev embedding of $H^1((0,L''))$ functions vanishing at $x=0$ into continuous functions on $[0,L'']$ vanishing at $x=0$.
\end{proof}

Let $u^0$ denote an eigenfunction of $Q(\nu;h)$ with eigenvalue $\lambda^0$ satisfying \eqref{itm:frobeigen} of Proposition \ref{prop:frobWKB}. From Lemma \ref{lem:frobmethod}, there is a unique smooth $y^0$ with $u^0 = x^{1/2+\nu}y^0$ and $y(0) \neq 0, \, (y^0)'(0) = 0$. As  indicated in Section \ref{subsect:idea}, let $u_\lambda$ denote the unique solution to the equation
\[
\left( h^2 D_x^2 + h^2(\nu^2 -1/4)x^{-2} + W - \lambda \right)u_\lambda = 0
\]
of the form $u_\lambda = x^{1/2+ \nu}w_\lambda$, where $w_\lambda$ is smooth and $w_\lambda(0) = w^0(0)$.

\subsection{Variation of parameters II}

The idea is to calculate $G'(\lambda^0)$ by variation of parameters as in Section \ref{subsect:variationofparam1}. The subsequent lemmas are analogues of those in Section \ref{subsect:variationofparam1}. 

\begin{lem} \label{lem:frobsecondsolution}
There exists $M > 0$ such that $|u^0(x)| > 0$ for $x\in [Mh^{1/2},L]$
\end{lem}
\begin{proof}
	The proof can be established exactly as Lemma \ref{lem:secondsolution}; the details are omitted.
\end{proof}

Define a complementary solution by
\[
v^0(x) = u^0(x) \int_{\pm M h^{1/2}}^x u^0(t)^{-2}\, dt,
\]
which therefore solves
\[
\begin{cases}
(Q(\nu;h)-\lambda^0 )v^0 = 0,\\ v^0(\pm Mh^{1/2}) = 0, \quad (v^0)'(\pm Mh^{1/2}) = u^0(\pm M h^{1/2})^{-1}.
\end{cases}
\]
Furthermore, $v^0(L) = h^{m+1}f(\nu;h)e^{\phi(r_\pm)/h}$, where
\[
f(\nu;h) \sim \sum_{j=0}^\infty f_{j}(\nu) h^j, \quad f_{0}(\nu) = {2\sqrt{W(L)}} e^{\phi(L)/h} L^{-1/2-\nu} a_0(L)^{-1}.
\]
Next is the analogue of Lemma \ref{lem:growth}.
\begin{lem} \label{lem:frobgrowth}
Let $K$ be a compact subinterval of $[0,\infty)$. Suppose that $y \in C^2(K)$ solves 
\[
(Q(\nu;h) - \lambda)x^{1/2+\nu}y= f
\]
where $0\leq \lambda \leq C_0h$ and $f \in x^{1/2+\nu}L^2(K)$. Then there exists $C>0$ depending on $K$ and $C_0$ such that
\begin{multline*}
e^{-\phi(x_1)/h}\left(h^{1/2}|y(x_1)| + |hy'(x_1) + \sqrt{W(x_1)}y(x_1)|\right) \\ \leq C e^{-\phi(x_0)/h} \left(h^{1/2}|y(x_0)| + |hy'(x_0) + \sqrt{W(x_0)}y(x_0)| \right) \\ + Ch^{-1} e^{-\phi(x_0)/h} \|e^{-\phi/h}x^{-1/2-\nu}f\|_{L^2(x_0,x_1)}
\end{multline*}
for $x_0, x_1 \in K$.
\end{lem}
\begin{proof}
Write 
\[
y = e^{\phi(x)/h}w; \quad f = x^{1/2+\nu} e^{\phi(x)/h}g,
\]
and set $\lambda = hE$ with $0 \leq E \leq C_0$. Then $w$ solves the equation
\[
h^2 \left( D_x^2 - (2\nu+ 1)x^{-1} \partial_x \right)w - h \left( 2\phi'(x) \partial_x + \phi''(x) + (2\nu + 1)x^{-1} \phi'(x) - E \right)w = g.
\]
A straightforward adaptation of the argument establishing Lemma \ref{lem:frobgrowth} finishes the proof.

\end{proof}

\begin{lem}
With $N(\nu;h)$ given by \eqref{itm:laguerrenorm} of Proposition \ref{prop:frobWKB},
\begin{enumerate} \itemsep4pt
\item $\mathcal W(\partial_\lambda u^0, u^0)(r_\pm) = h^{-1-\nu+2m} N(\nu;h)^2$,
\item $\mathcal W(\partial_\lambda u^0, v^0)(r_\pm) = \mathcal{O}(h^{-K})$ for some $K>0$.
\end{enumerate}
Consequently,
\[
\partial_\lambda u^0 (L) = -h^{-1-\nu+2m} N(\nu;h)^2\, v^0(L) + \mathcal{O}\left( h^{-K} e^{-\phi(x)/h} \right).
\]
\end{lem}
\begin{proof}
$(1)$ First note that $\mathcal{W}(x^{1/2+\nu}f_1,x^{1/2+\nu}f_2)(x) = x^{1+2\nu}\mathcal{W}(f_1,f_2)$, and that
\begin{equation} \label{eq:froblambda}
\partial_\lambda u^0 = x^{1/2+\nu} \partial_\lambda y^0; \quad \partial_\lambda y^0(0) = (\partial_\lambda y^0)'(0) = 0.
\end{equation}
Therefore $\mathcal{W}(\partial_\lambda u^0, u^0)(0)=0$, so by the same argument as in Lemma \ref{lem:wronskians1},
\[
\mathcal{W}(\partial_\lambda u^0, u^0)(L) =  h^{-2} \int_0^L u^0(t)^2 \,dt = h^{-1-\nu+2m} N(\nu;h)^2.
\]

$(2)$ Similarly,
\[
\mathcal W(\partial_\lambda u^0, v^0)(L) = \mathcal W(\partial_\lambda u^0, v^0)(Mh^{1/2}) + h^{-2} \int_{Mh^{1/2}}^{L} v^0(t) u^0(t) \,dt.
\]
Now $\partial_\lambda u_\lambda$ solves
\[
(Q(\nu;h) - \lambda) \partial_\lambda u_\lambda = u_\lambda,
\]
so it follows from \eqref{eq:froblambda} and Lemma \ref{lem:frobgrowth} that $\partial_\lambda u^0(x) = \mathcal{O}(h^{-K}e^{\phi(x)/h})$.

Also $u^0(x) v^0(x) = \mathcal{O}(h^{-K})$ from Lemma \ref{lem:frobsecondsolution}. The rest of the proof goes through just as in Lemma \ref{lem:wronskians1}.
\end{proof}

\begin{prop} \label{lem:mainlemma2}
Fix an integer $m\geq 0$.
\begin{enumerate}
\item With $a(x,h)$ given by Proposition \ref{prop:frobWKB}, we have that
\[
G(\lambda^0) = h^{-m} L^{1/2+\nu}a(L,h)e^{-\phi(L)/h} + \mathcal{O}\left(h^\infty e^{-\phi(L)/h}\right).
\]

\item The derivative $G'(\lambda^0)$ is nonzero, and 
\[
G'(\lambda^0)^{-1} = -h^{-\nu - m}  p(\nu;h) e^{-\phi(L)/h},
\]
where  $p(\nu;h)$ admits the same asymptotic expansion as $N(\nu;h)^{-2} f(\nu;h)^{-1}$, so that 
\[
p(\nu;h) \sim \sum_{i=0}^{\infty} p_j(\nu) h^j, \quad p_0(\nu) = N_0(\nu)^{-2} f_0(\nu)^{-1} = \frac{4\sqrt{W(L)}}{\Gamma(1+m+\nu) m!} L^{1/2+\nu} a_0(L).
\]

\item \label{itm:frobD2G} Given $C_0> 0$, suppose that $0 \leq \lambda \leq C_0h$. Then 
\[
G''(\lambda) = \mathcal{O}\left(h^{-K}e^{\phi(L)/h}\right).
\]

\end{enumerate}
\end{prop}

\begin{proof}
It remains to prove \ref{itm:frobD2G}, which follows from Lemma \ref{lem:frobgrowth}.
\end{proof}

\begin{proof} [Proof of Theorem \ref{theo:maintheo2}]
The proof is the same as that of Theorem \ref{theo:maintheo1}, or in fact slightly simpler since $G$ is scalar valued in this case. Following the same argument, we find that
\[
\lambda^\Lambda - \lambda^0 = h^{-\nu-2m} e^{-2\phi(L)/h} s(\nu;h),
\]
where $s(\nu;h) \sim \sum_{j=0}^\infty s_j(\nu) h^j$ admits the same asymptotic expansion as 
\[
L^{1/2+\nu}N(\nu;h)^{-2} f(\nu;h)^{-1}a(L,h),
\] and hence
\[
s_0(\nu) =  \frac{4\sqrt{W(L)}}{\Gamma(1+m+\nu) m!} L^{1+2\nu} a_0(L)^2.
\]

\end{proof}

\section*{Acknowledgements}
I would like to thank Semyon Dyatlov for an introduction to the techniques used in this paper and for pointing out the applicability of the WKB method for more general potentials. I would also like to thank Maciej Zworski for suggesting that this method could be applied to confined eigenvalue problems, and for several useful discussions.

\end{document}